\documentclass[a4paper]{amsart}
\usepackage{amsthm,amsmath,amssymb,mathrsfs}
\usepackage[latin1]{inputenc}
\usepackage[british]{babel}
\usepackage{mathrsfs}
\usepackage{url}
\usepackage{mathtools}

\usepackage{tikz} 
\usetikzlibrary{matrix,arrows,calc,decorations.pathreplacing,fit}
\usetikzlibrary{fadings}
\usepackage{tikz-cd}

\numberwithin{equation}{section}

\newtheorem{theorem}[equation]{Theorem}
\newtheorem{lemma}[equation]{Lemma}

\newtheorem{proposition}[equation]{Proposition}

\theoremstyle{definition}

\theoremstyle{remark}
\newtheorem{remark}[equation]{Remark}

\DeclareMathOperator{\cha}{char}

\DeclareMathOperator{\Gal}{Gal}

\DeclareMathOperator{\Isom}{Isom}

\def \mono  {\hookrightarrow}

\def \epi   {\twoheadrightarrow}

\def \bij   {\xrightarrow{1:1}}

\makeatletter
\newcommand*{\relrelbarsep}{.386ex}
\newcommand*{\relrelbar}{%
  \mathrel{%
    \mathpalette\@relrelbar\relrelbarsep
  }%
}
\newcommand*{\@relrelbar}[2]{%
  \raise#2\hbox to 0pt{$\m@th#1\relbar$\hss}%
  \lower#2\hbox{$\m@th#1\relbar$}%
}
\providecommand*{\rightrightarrowsfill@}{%
  \arrowfill@\relrelbar\relrelbar\rightrightarrows
}
\providecommand*{\xrightrightarrows}[2][]{%
  \ext@arrow 0359\rightrightarrowsfill@{#1}{#2}%
}

\newcommand{\restr}[2]{{#1}\raise-.5ex\hbox{\ensuremath|}_{#2}}     

\newcounter{subenvcounter}
\newenvironment{subenv}{%
 \begin{list}
  {\em (\arabic{subenvcounter})}
  {\setlength{\leftmargin}{20pt}
   \setlength{\rightmargin}{0pt}
   \setlength{\itemindent}{0pt}
   \setlength{\labelsep}{5pt}
   \setlength{\labelwidth}{13pt}
   \setlength{\listparindent}{\parindent}
   \setlength{\parsep}{0pt}
   \setlength{\itemsep}{0pt}
   \setlength{\topsep}{-\parskip}
   \usecounter{subenvcounter}}}
  {\end{list}}

\newcounter{asslistcounter}
\newenvironment{assertionlist}{
 \begin{list}
  {\upshape (\alph{asslistcounter})}
  {\setlength{\leftmargin}{18pt}
   \setlength{\rightmargin}{0pt}
   \setlength{\itemindent}{0pt}
   \setlength{\labelsep}{5pt}
   \setlength{\labelwidth}{13pt}
   \setlength{\listparindent}{\parindent}
   \setlength{\parsep}{0pt}
   \setlength{\itemsep}{0pt}
   \setlength{\topsep}{-.5\parskip}
   \usecounter{asslistcounter}}}
  {\end{list}}

\def \NN {\mathbb{N}}

\def \QQ {\mathbb{Q}}
\def \RR {\mathbb{R}}

\def \ZZ {\mathbb{Z}}

\def \Acal {\mathcal{A}}
\def \Bcal {\mathcal{B}}

\def \Mcal {\mathcal{M}}

\def \afr {\mathfrak{a}}

\def \Gscr {\mathscr{G}}

\def \hbar {\bar{h}}

\def \Wtilde {\tilde{W}}
\def \Xtilde {\tilde{X}}

\author{Paul Hamacher and Eva Viehmann}
\address{Technische Universit\"at M\"unchen\\
Fakult\"at f\"ur Mathematik - M11 \\ Boltzmannstr. 3\\
85748 Garching bei M\"unchen\\
Germany}
\email{hamacher@ma.tum.de, viehmann@ma.tum.de}
\date{}
\thanks{The authors were partially supported by ERC Consolidator Grant 770936:\ NewtonStrat.}

\title{Finiteness properties of affine Deligne-Lusztig varieties}
\def \Flag {{\mathscr{F}\hspace{-0.5mm}\ell}}
\def \unif {{\epsilon}}

\newcommand{\cl}[1]{\mkern 1.5mu\overline{\mkern-1.5mu#1\mkern-1.5mu}\mkern 1.5mu}

\begin{document}
 \begin{abstract}
  {Affine Deligne-Lusztig varieties are closely related to the special fibre of Newton strata in the reduction of Shimura varieties or of moduli spaces of $G$-shtukas. In almost all cases, they are not quasi-compact. In this note we prove basic finiteness properties of affine Deligne-Lusztig varieties under minimal assumptions on the associated group. We show that affine Deligne-Lusztig varieties are locally of finite type, and prove a global finiteness result related to the natural group action. Similar results have previously been known for special situations.}
 \end{abstract}
 \maketitle

 \section{Introduction}

 Let $F$ be a local field, $O_F$ its ring of integers, and $k_F=\mathbb{F}_q$ its residue field, a finite field of characteristic $p$. We denote by $L$ the completion of the maximal unramified extension of $F$, and by $O_L$ its ring of integers. Then the residue field $k$ of $L$ is an algebraic closure of $\mathbb{F}_q$. We denote by $\unif$ a uniformizer of $F$, which is then also a uniformizer of $L$. Let $\sigma$ be the Frobenius of $k$ over $k_F$ and also of $L$ over $F$. We denote by $I$ the inertia group of $F$.

 We consider a smooth affine group scheme $\Gscr$ over $O_F$ with reductive generic fibre. Let $P = \Gscr(O_L)$ and let $G=\Gscr_F$.

 We denote by $\Flag_\Gscr$ the base change to $k$ of the affine flag variety (over $k_F$) associated with $\Gscr$ as in \cite[\S~1.c]{PappasRapoport:AffineFlag} and \cite[Def.~9.4]{BhattScholze:AffGr}. In particular, $\Flag_\Gscr$ is a sheaf on the fpqc-site of $k$-schemes ($\cha F = p$) resp.~of perfect $k$-schemes ($\cha F = 0$) with
 \[
  \Flag_\Gscr(k) = G(L)/P,
 \]
which is representable by an inductive limit of finite type schemes ($\cha F = p$) resp.~of perfectly of finite type schemes ($\cha F = 0$); see \cite[Thm.~1.4]{PappasRapoport:AffineFlag},\cite[Cor.~9.6]{BhattScholze:AffGr}. Hence we can define an underlying topological space of $\Flag_\Gscr$, which is Jacobson. This means that by mapping a subset of $\Flag_\Gscr$ to its intersection with the subset of closed points $\Flag_\Gscr(k) \subset \Flag_\Gscr$ we obtain a bijection between the open subsets of $\Flag_\Gscr$ and the open subsets of $\Flag_\Gscr(k)$ (same for closed and for locally closed subsets). Moreover, being a base change from $k_F$, we have an action of $\sigma$ on $\Flag_\Gscr$.

 To define affine Deligne-Lusztig varieties we fix an element $b\in G(L)$ and a locally closed subscheme $Z$ of the loop group $LG$ which is stable under $P$-$\sigma$-conjugation. Then we consider the functor on reduced $k$-schemes resp.~reduced perfect $k$-schemes with
$$X_Z(b)(S)=\{g\in \Flag_{\Gscr}(S)\mid g_x^{-1}b\sigma(g_x)\in Z(\kappa_x)\text{ for every geometric point } x\text{ of }S\}.$$ 

\begin{remark}
 The functor $X_Z(b)$ defines a locally closed reduced sub-indscheme of $\Flag_{\Gscr}$: Consider the functor $\tilde X_Z(b)$ on reduced $k$-schemes resp.\ perfect $k$-schemes with
$$\tilde X_Z(b)(S)=\{g\in LG(S)\mid g_x^{-1}b\sigma(g_x)\in Z(\kappa_x)\text{ for every geometric point } x\text{ of }S\}.$$ Then $\tilde X_Z(b)$ is the inverse image of $Z$ under the morphism $LG\rightarrow LG$ with $g\mapsto g^{-1}b\sigma(g)$. Since $Z$ is locally closed, also $\tilde X_Z(b)$ defines a locally closed reduced sub-ind-scheme of $LG$. Furthermore, $X_Z(b)$ is the image of $\tilde X_Z(b)$ under the quotient map $LG\rightarrow \Flag_{\Gscr}$, which is an $L^+\Gscr$-torsor. Hence it is again a locally closed sub-ind-scheme.
\end{remark}

 Let $J_b$ be the reductive group over $F$ whose $R$-valued points for any $F$-algebra $R$ are given by $$J_b(R)=\{g\in G(R\otimes_F L)\mid gb=b\sigma(g)\}.$$  Then for every $Z$ there is a natural action of $J_b(F)$ on $X_Z(b)$ given by left multiplication. Our main result is

 \begin{theorem}\label{thm finiteness}
  Assume in addition that $Z$ is bounded (see Section~\ref{sec2} for the definition of boundedness).
  \begin{subenv}
   \item $X_Z(b)$ is a scheme which is locally of finite type in the case that $\cha F = p$ and locally perfectly of finite type in the case $\cha F = 0$.
   \item The action of $J_b(F)$ on the set of irreducible components of $X_Z(b)$ has finitely many orbits.
  \end{subenv}
 \end{theorem} 
This theorem is related to the fact that they are the underlying reduced subscheme of moduli spaces of local $G$-shtukas and to the general expectation for the arithmetic case that (at least in the minuscule case) affine Deligne-Lusztig varieties are the reduction modulo $p$ of integral models of local Shimura varieties. Their cohomology is conjectured to decompose according to the local Langlands and Jacquet-Langlands correspondences. In order to be able to apply the usual methods, one needs the cohomology groups to be finitely generated $J_b(F)$-representations, and thus the ``infinite level'' cohomology groups to be admissible. This follows from the above theorem by a formal argument once the integral model is constructed (see for example \cite[Thm.~4.4]{Mieda:CompRZ}, \cite[Prop.~6.1]{RapoportViehmann:LocShVar}).

Many particular cases of the theorem have been considered before. For the particular case of affine Deligne-Lusztig varieties arising as the underlying reduced subscheme of a Rapoport-Zink moduli space of $p$-divisible groups with additional structure of PEL type, questions as in Theorem~\ref{thm finiteness} have been considered by several people. A recent general theorem along these lines is shown by Mieda \cite{Mieda:CompRZ}. Also, the (rare) cases where an affine Deligne-Lusztig variety is even of finite type have been classified, compare \cite[Prop.~4.13]{Goertz:Survey}.

In the case where $\Gscr$ is reductive over $O_F$ and $Z$ is a single $P$-double coset, a complete description of the set of $J_b(F)$-orbits of irreducible components of $X_{Z}(b)$ is known. The present work was motivated by our own results in this direction in \cite{HamacherViehmann:ADLVIrredComp}. Recently, complete descriptions were given by Zhou and Zhu \cite{ZhouZhu:IrredCompADLV} and by Nie \cite{Nie:ADLVIrredComp}.

The main tool to prove Theorem \ref{thm finiteness} is to relate the claimed finiteness statements to finiteness properties of certain subsets of the extended Bruhat-Tits building of $G$, using previous work of Cornut and Nicole \cite{CornutNicole:CrysBuildings}.\\

{\it Acknowledgement.} We are grateful to G. Prasad for pointing out some of his work on Bruhat-Tits theory to us. We thank the referee for his/her helpful comments.

\section{Reduction to the parahoric case}

 As a first step, we reduce to the case that $\Gscr$ is a parahoric group scheme. While most assertions in the following still hold true in the general setup, the assertion that $\Gscr$ is parahoric will simplify the proofs and the notation.
 
 By the fixed point theorem \cite[2.3.1]{Tits:padicRedGpCorvalis} the group $P \rtimes \langle \sigma \rangle$ has a fixed point $x$ in the extended Bruhat-Tits building of $G_L$. We refer to the subsequent section for the relation between the extended Bruhat-Tits building and the ``classical'' Bruhat-Tits building. 
  By definition the stabiliser $P_x$ of $x$ is $\sigma$-stable and contains $P$.  We denote by $\Gscr_{x}$ the corresponding group scheme over $O_F$ in the sense of \cite[1.9]{Prasad:UnramDescent} and \cite[2]{Prasad:TameDescent}.

 \begin{lemma}
  The fpqc quotient $L^+\Gscr_{x}/L^+\Gscr$ is representable by a finitely presented (resp.~perfectly finitely presented) scheme.
 \end{lemma}
 \begin{proof}
  We denote $P_{{x},n} \coloneqq \ker(\Gscr_{x}(O_L) \to \Gscr_{x}(O_L/\unif^n))$. Since the $P_{{x},n}$ form a neighbourhood basis of the unit element in $G(L)$ we have $P_{{x},n} \subset P$ for some $n$. Thus the positive loop group $L^+P$ contains the kernel of the reduction map into the truncated positive loop group $L^+P_{x} \to L^+_{n} P_{x}$. Indeed, we have just shown that this is true on geometric points and the kernel is an infinite dimensional affine space by Greenberg's structure theorem \cite[p.~263]{Greenberg:SchLocRings2}, thus in particular reduced. Hence we get $L^+\Gscr_{x}/L^+\Gscr \cong L^+_n\Gscr_{x}/L^+_n\Gscr$. Since the latter is a quotient of linear algebraic groups over $k_F$, the claim follows.
 \end{proof}

Since $LG \to \Flag_{\Gscr}$ is an $L^+\Gscr$-torsor, we get that $\Flag_{\Gscr}$ is \'etale locally isomorphic to $\Flag_{\Gscr_{x}} \times L^+\Gscr_{x}/L^+\Gscr$. In particular, the canonical projection $\Flag_\Gscr \to  \Flag_{\Gscr_{x}}$ is relatively representable and of finite type. Thus Theorem~\ref{thm finiteness} holds true for $\Gscr$ if and only if it is true for $\Gscr_{x}$, as it is enough to prove the theorem after enlarging $Z$ so that becomes stable under $P_x$-$\sigma$-conjugation. Let $\Gscr_x^\circ \subset \Gscr_x$ be the parahoric group scheme associated to $x$. Repeating the argument above, we see that it suffices to prove Theorem~\ref{thm finiteness} for $\Gscr_x^\circ$ instead of $\Gscr$.
 
 Therefore we can (and will) assume from now on that $\Gscr$ is a parahoric group scheme.

\section{Some properties of Bruhat-Tits buildings}\label{sec2}

 We consider the following group theoretical setup. Let $S_0 \subset G$ be a maximal $L$-split torus defined over $F$, let $T_0$ be its centraliser and let $N_0$ be the normaliser of $T_0$ in $G$. Then $T_0$ is a torus because $G$ is quasi-split over $L$. Thus $W = N_0(L)/T_0(L)$ is the relative Weyl group of $G$ over $L$. We denote by $P_{T_0}$ the unique parahoric subgroup of $T_{0}$. The extended affine Weyl group is defined as
 \[
  \Wtilde \coloneqq N_0(L)/P_{T_0} \cong X_\ast(T_0)_{\Gal(\cl{L}/L)} \rtimes W,
 \]
where $X_\ast(T_0)_{\Gal(\cl{L}/L)}$ denotes the group of Galois convariants of $X_\ast(T_0)$ over $L$. We may choose $S_0$ such that $P$ stabilises a facet in the apartment of $S_0$ and denote $\Wtilde^P = (N_0(L) \cap P)/P_{T_0} \subset \Wtilde$. By \cite[Appendix, Prop.~9]{PappasRapoport:AffineFlag} the embedding $N_0 \mono G$ induces a bijection
 \begin{equation} \label{eq-Iwasawa-decomposition}
  \Wtilde^P \backslash \Wtilde / \Wtilde^P \bij P\backslash G(L) / P.
 \end{equation}
  We call a subset $\Xtilde \subset G(L)$ bounded if it is contained in a finite union of $P$-double cosets. The bounded subsets form a bornology on $G(L)$, which does not depend on the choice of $P$. 

 Let $\Bcal^e(G,L)$ be the extended Bruhat-Tits building of $G$ over $L$, that is
 \[
  \Bcal^e(G,L) = \Bcal(G,L) \times V_0(G,L)
 \]
 where $\Bcal(G,L)$ is the ``classical'' Bruhat-Tits building of $G$ and $V_0(G,L) \coloneqq X_\ast(G^{\rm ab})_\RR^{\Gal(\cl{L}/L)}  \cong  X_\ast(Z(G))_\RR^{\Gal(\cl{L}/L)} $ with $Z(G)$ denoting the center of $G$. The extended apartment $\Acal^e(S,G) \subset \Bcal^e(G,L)$ of a maximal $L$-split torus $S$ is defined as $\Acal(S,G;L) \times V_0(G,L)$ where $\Acal(S,G;L)$ denotes the apartment of $S$. We recall from Landvogt \cite[\S~1.3]{Landvogt:Functoriality} that $\Bcal^e(G,L)$ is a polysimplicial complex with a metric $d$ and a $G(L) \rtimes \langle \sigma \rangle$-action by isometries. Moreover, one can canonically identify $\Bcal^e(G,F)$ with the set of $\sigma$-invariants $\Bcal^e(G,L)^{\langle \sigma \rangle}$. 
 
 We consider the canonical map
 \[
  i\colon G(L) \to \Isom (\Bcal^e(G,L)),
 \]
 where $\Isom (\Bcal^e(G,L))$ denotes the space of self-isometries of $\Bcal^e(G,L)$. A set $M \subset \Isom (\Bcal^e(G,L))$ is called bounded if for some (or equivalently every) non-empty bounded set $A \subset \Bcal^e(G,L)$ the set $\{f(x) \mid f \in M, x \in A\} \subset \Bcal^e(G,L)$ is bounded. We have the following statement about the compatibility of bornological structure.

 \begin{proposition}[{\cite[Prop.~4.2.19]{BruhatTits:RedGp2}}] \label{prop-bornology-comparison}
  A subset  $\Xtilde \subset G(L)$ is bounded if and only if its image under $i$ is.
 \end{proposition}  
 
 We consider the following maps between extended Bruhat-Tits buildings. Let $f\colon G \to G'$ be a morphism of reductive $F$-groups. A $G(L)$-equivariant map $g\colon\Bcal^e(G,L) \to \Bcal^e(G',L)$ is called \emph{toral} if for every maximal $L$-split torus $S \subset G_L$ there exists a maximal $L$-split torus $S' \subset G'_L$ such that $f(S) \subset S'$ and $g$ restricts to an $X_\ast(S)_\RR$-translation equivariant map between the apartments of $S$ and $S'$. In \cite{Landvogt:Functoriality}, Landvogt proves that there always exists a $G(L) \rtimes \langle \sigma \rangle$-invariant toral map, which becomes an isometry after normalising the metric on $\Bcal^{e}(G',L)$. However, this map depends on an auxiliary choice. We give a precise formulation of the result in the form and context that we need later on. For this consider the fixed element $b\in G(L)$ and denote by $\nu_b \in X_\ast(G)_\QQ$ the Newton point of $b$ (see \cite[\S~4]{Kottwitz:GIsoc1} for its precise definition). We fix an integer $s \gg 0$ such that $s \cdot \nu_b \in X_\ast(G)$.  
 Denote by $M_b \subset G$ the Levi subgroup centralising $\nu_b$ (and thus $s\cdot\nu_b$). Then $J_{b}$ is the inner form of $M_b$ obtained by twisting the action of the Frobenius by $b$. We can thus use the following result to relate the buildings of $G$ and $J_b$. A similar result is also shown in \cite{CornutNicole:CrysBuildings}.
 
 \begin{proposition}[{\cite[Prop.~2.1.5]{Landvogt:Functoriality},\cite[Lemme~5.3.2]{Rousseau:Thesis}}] \label{prop-buildings-levi}
  Let $f\colon M_{b} \mono G$. Then there exists a toral $M_b(L) \rtimes \langle \sigma \rangle$-equivariant injective map
  \[
   f_\ast\colon \Bcal^e(M_b,L) \to \Bcal^e(G,L).
  \]
  Moreover, $f_\ast$ is injective and unique up to translation by an element of $V_0(G,L)^{\langle \sigma \rangle}$. In particular, its image is the same for every choice of $f_{\ast}$ and equal to $\Bcal^e(G,L)^{ (s\cdot\nu_b)(O_L^\times) }$. After a suitable normalisation of the metric on $\Bcal^e(G,L)$, this map becomes an isometry.
 \end{proposition}  
 
\begin{remark}\label{rem23b} Since $J_{b,L} \cong M_{b,L}$, we obtain an identification of $\Bcal^{e}(J_b,L)$ with $\Bcal^{e}(G,L)^{(s\cdot\nu_b)(O_L^\times)}$. However, since $J_b$ is an inner twist of a Levi subgroup of $G$, this identification will not respect the action of the Frobenius in general. In order to distinguish it from the action on $\Bcal^e(G,L)$, we denote the Frobenius action on $\Bcal^{e}(J_b,L)$ (and $J_b(L)$) by $\sigma_b$. More explicitely, we have $\sigma_b = \restr{b\sigma}{\Bcal(J_b,L)} \times \restr{\sigma}{V_0(J_b,L)}$. Indeed, by \cite[Lemma~3.3.1]{Landvogt:Thesis}, the Frobenius action on the ``classical'' Bruhat-Tits building $\Bcal(J_b,L)$ is uniquely determined by the equation $\sigma_b(j.x) = \sigma_b(j).\sigma_b(x) = (b \sigma(j) b^{-1}).\sigma_b(x)$ and thus has to be equal to $b\sigma$. It follows from the explicit description in \cite[(3.3.2)]{Landvogt:Thesis}, that the Frobenius action on $V_0(J_b,L)$ remains the same.
 \end{remark}

  Now assume that we have an embedding of reductive groups $f\colon G \mono G'$. The following statement is the main result of \cite{Landvogt:Functoriality}.
  
  \begin{proposition}[{\cite[Thm.~2.2.1]{Landvogt:Functoriality}}] \label{prop-buildings-functoriality}
   There exists a $G(L) \rtimes \langle \sigma \rangle$-invariant toral map $f_\ast\colon \Bcal^{e}(G,L) \to \Bcal^{e}(G',L)$. Furthermore the metric on $\Bcal^{e}(G,L)$ can be normalised in a way such that $f_\ast$ becomes isometrical.
  \end{proposition}
  
  To simplify the notation, we identify $G$ with its image in $G'$. Now $b$, considered as element of $G'$, induces a group $J'_b$ which is an inner form of the centraliser of $\nu_b$ in $G'$. Since $f_\ast$ preserves the fixed points of $\nu_b(O_L^\times)$, we obtain a commutative diagram by Proposition~\ref{prop-buildings-levi} and Remark \ref{rem23b},
 \begin{equation}\label{thediagram}
   \begin{tikzcd}
    \Bcal^e(J_b,L) \arrow[hook]{r}{f_\ast} \arrow[hook]{d} & \Bcal^e(J'_{b},L) \arrow[hook]{d} \\
    \Bcal^{e}(G,L) \arrow[hook]{r}{f_\ast} &\Bcal^{e}(G',L).
   \end{tikzcd}
 \end{equation}
  
 \begin{lemma} \label{lemma-buildings-Frobenius-equivariance}
  The restriction $\restr{f_\ast}{\Bcal^e(J_b,L)}$ is $\sigma_b$-equivariant.
 \end{lemma}
 A related statement is \cite{CornutNicole:CrysBuildings}, 3.5. For the convenience of the reader, we provide the details of the proof here.
 \begin{proof}
  We denote by $\sigma'_b$ the canonical Frobenius action on $\Bcal^e(J'_b,L)$. Note that the action of $b\sigma$ and the actions of $\sigma_b, \sigma_b'$ differ by the translations $t_b,t_b'$ induced by the action of $b$ on $V_0(J_b,L)$ and  $V_0(J_b',L)$ respectively. Since $f_\ast$ is $b\sigma$-equivariant, it suffices to show that $f_\ast \circ t_b = t_b' \circ f_\ast$. 

  To prove this, consider the composition of $f_\ast$ with the canonical projection $\Bcal^e(J_b',L) \epi V_0(J_b',L)$. We claim that this map factors through $V_0(J_b,L)$. This can be checked on extended apartments. Let $S \subset J_b, S' \subset J_b'$ be maximal split tori over $L$ with $f(S) \subset S'$ and $f_\ast(\Acal^e(S,J_b;L)) \subset \Acal^e(S',J'_b;L)$. For the intersections with the derived groups of $G,G'$ we have $S^{\rm der} \subset S'^{\rm der}$. Hence the composition $\Acal^e(S,J_b;L) \mono \Acal^e(S',J'_b;L) \epi V_0(J_b';L)$ is $S^{\rm der}(L)$-invariant and thus factors through $\Acal^e(S,J_b;L)/\Acal^e(S^{\rm der},J_b^{\rm der};L) = V_0(J_b,L)$. 
 
 Thus we obtain a commutative diagram
 \begin{center}
  \begin{tikzcd}
   \Bcal^e(J_b,L) \arrow[hook]{r}{f_\ast} \arrow[two heads]{d}{p} & \Bcal^e(J_b',L) \arrow[two heads]{d}{p'} \\
   V_0(J_b,L) \arrow[hook]{r}{f_\ast^{\rm ab}} & V_0(J_b',L)
  \end{tikzcd}
 \end{center}
  Since $p,p'$ and $f_\ast$ commute with the action of $b$, so does $f_\ast^{\rm ab}$. Thus $f_\ast^{\rm ab} \circ t_b = t_{b}' \circ f_\ast^{\rm ab}$, proving $f_\ast \circ t_b = t_{b}' \circ f_\ast$.
 \end{proof}

\section{Boundedness properties on the affine flag variety}\label{sec3}

 We denote by $w_G\colon G(L) \to \pi_1(G)_I$ the Kottwitz homomorphism. For any subset $X \subset G(L)$ and $\omega \in \pi_1(G)_I$, we define $X^\omega \coloneqq X \cap w_G^{-1}(\{\omega\})$. We remark that by \cite[Thm.~5.1]{PappasRapoport:AffineFlag} and \cite[Prop.~1.21]{Zhu:Satake} the connected components of $\Flag_\Gscr$ are precisely the subsets of the form $\Flag_\Gscr^\omega$.

 For further considerations, it will be useful to fix a presentation of $\Flag_\Gscr^{\rm red}$ as a limit of schemes. For any $w \in \Wtilde^P \backslash \Wtilde / \Wtilde^P$ we denote by
 \begin{align*}
  S_w^\circ &\coloneqq PwP/P \\
  S_w &\coloneqq \bigcup_{w' \leq w} S_{w'}
 \end{align*}
the Schubert cell and the Schubert variety associated with $w$, respectively. Here, $\leq$ denotes the Bruhat order on $\widetilde W$ induced by any fixed choice of an Iwahori subgroup of $P$. By \cite[\S~8]{PappasRapoport:AffineFlag} and \cite[Thm.~9.3]{BhattScholze:AffGr} each Schubert variety (resp.~cell) is a closed (resp.~locally-closed) quasi-compact subscheme of $\Flag_\Gscr$, which is of finite type in the case $\cha F = p$ and perfectly of finite type in the case $\cha F = 0$. Note that by (\ref{eq-Iwasawa-decomposition}), we have that $\Flag_\Gscr = \bigcup S_w^\circ$ is a decomposition into locally closed subsets, hence we can write $\Flag_\Gscr^{\rm red} = \varinjlim S_w$.

 We equip $\Flag_\Gscr(k)$ with the bornology induced by the canonical projection $G(L) \epi G(L)/P = \Flag_\Gscr(k)$, that is a subset $X \subset \Flag_\Gscr(k)$ is bounded, if it is contained in a finite union of Schubert varieties. We obtain the following geometric characterisation of bounded subsets.

\begin{lemma} \label{lemma-bornology-on-Fl}
 A subset $X \subset \Flag_\Gscr(k)$ is bounded if and only if it is relatively quasi-compact (i.e.\ contained in a quasi-compact subset). In this case $X$ is even quasi-compact itself.
\end{lemma}
\begin{proof}
 Since the $S_w$ are quasi-compact, any bounded subset of $\Flag_\Gscr$ is relatively quasi-compact. The $S_w$ are Noetherian, thus their subsets are quasi-compact themselves. 
 
  On the other hand, assume that $X$ is not bounded. We prove that $X$ is not quasi-compact by constructing an infinite discrete closed subset $Y \subset X$.
 By definition, the set $ T \coloneqq \{w \in \widetilde{W} \mid X \cap S_w^\circ \not= \emptyset \}$ is infinite. For each $w \in T$, choose an element $x_w \in  X \cap S_w^\circ$. Then $Y \coloneqq \{x_w \mid w\in T\}$ is infinite and discrete. Its intersection with every $S_w$ for $w\in\widetilde W$ is closed, hence $Y$ is closed.
 \end{proof}

\begin{lemma} \label{lemma-representability-in-Fl}
 Let $X \subset \Flag_\Gscr$ be a locally closed reduced sub-ind-scheme.
 Then $X$ is a scheme if and only if every point of $X(k)$ has an open neighbourhood which is bounded as subset of $\Flag_\Gscr(k)$. In this case $X$ is locally of finite type if $\cha F =p$, respectively locally of perfectly finite type if $\cha F  = 0$.  
\end{lemma}
\begin{proof}
 The ``only if'' direction follows from the previous lemma because every point of a scheme has a quasi-compact open neighbourhood.

 To prove the ``if'' direction, we may assume that $X$ is bounded, since its representability is a Zariski-local property. Then the embedding $X(k) \mono \Flag_\Gscr$ factors through some finite union of Schubert varieties by the previous lemma, in particular $X(k)$ is a locally closed subvariety of this union. Since the Schubert varieties are  (perfectly) of finite type, so is $X$.
\end{proof}

\begin{remark}
 The analogous assertions of Lemmas~\ref{lemma-bornology-on-Fl} and \ref{lemma-representability-in-Fl} in $LG(k)$, the loop group of $G$, also hold true (with the exception of the last statement of Lemma~\ref{lemma-representability-in-Fl}). Indeed, since a set $X \subset G(L)$ is bounded if and only if $X \cdot P$ is bounded, it suffices to prove the assertion in the case that $X$ is right $P$-invariant. Then the claim follows from the above lemmas since $LG \to \Flag_\Gscr$ is an $L^+ \Gscr$-torsor and thus relatively representable and quasi-compact.
\end{remark}

\section{Affine Deligne Lusztig varieties}

We now prove that the first part of Theorem~\ref{thm finiteness} implies the second. By Lemma~\ref{lemma-representability-in-Fl} together with the first part of the theorem, its second assertion is equivalent to the following proposition, which we prove below.
  
 \begin{proposition}\label{prop Jb-orbits}
Let  $Z$ a bounded subset of $G(L)$ and denote
 \[
  \Xtilde_Z(b) \coloneqq \{g \in G(L) \mid g^{-1}b \sigma(g) \in Z \}.
 \]
 Then there exists a bounded subset $\Xtilde_0 \subset \Xtilde_Z(b)$ such that $\Xtilde_Z(b) =  J_b(F) \cdot \Xtilde_0$.
\end{proposition}

For the proof of the proposition we need some preparation.

\begin{lemma}\label{rem42}
The $\sigma$-conjugacy class of $b\in G(L)$ has a decent representative for which $\Bcal^e(J_b,F)\cap \Bcal^e(G,F)\neq \emptyset$ (viewed as subspaces of $\Bcal^e(G,L)$).
\end{lemma}
Here, an element $b\in G(L)$ is called decent if there is a natural number $s$ with $(b\sigma)^s=s\nu(\epsilon)\sigma^s$.
\begin{proof}
In Remark~\ref{rem23b} we identified the extended Bruhat-Tits building $\Bcal^e(J_b,L)$ with $\Bcal^e(M_b,L)$. We fix a maximal $L$-split torus $S \subset M_b$, denote by $T$ its centraliser and by $\Wtilde_{M_b}$ the associated extended affine Weyl group of $M_b$. Since any reductive group over $F$ is residually quasi-split by \cite[Thm.~4.1]{BruhatTits:RedGp3}, there exists a $\sigma$-stable alcove $\afr$ in $\Acal(S,M_b,L)$. 
 The Kottwitz homomorphism maps the stabiliser $\Omega \subset \Wtilde_{M_b}$ of $\afr$ isomorphically onto $\pi_1(G)_I$. Since any basic $\sigma$-conjugacy class is uniquely determined by its Kottwitz point, we may assume that $b$ (after replacing it by a $M_b(L)$-$\sigma$-conjugate if necessary) is a representative in $M_b(L)$ of an element of $\Omega$. By \cite[Lemma~2.2.10]{Kim:CenLeaf} we may assume this representative to be decent. It now follows from the explicit description of $\sigma_b$ in Remark~\ref{rem23b} that we may take $p_0 \coloneqq (p_b,p_v)$ where $p_b \in \Bcal(M_b,L)$ is the barycenter of $\afr$ and $p_v \in V_0(M_b,L)$ is any point fixed by $\sigma$. Then $p_0\in \Bcal^e(J_b,F)\cap \Bcal^e(G,F).$
\end{proof}

Thus after replacing $b$ by a $\sigma$-conjugate if necessary, we fix $p_0\in \Bcal^e(J_b,F) \cap \Bcal^e(G,F)$. In order to relate the bornologies on $G(L)$ and on $\Bcal^e(G,L)$ directly, we consider the map
\[
 \iota\colon G(L) \to \Bcal^e(G,L), g \mapsto g.p_0.
\]
By the choice of $p_0$, the map $\iota$ is $G(L) \rtimes \langle \sigma \rangle$-equivariant and the restriction to $J_b(L)$ is moreover $\sigma_b$-equivariant, cf. Remark \ref{rem23b}. By Proposition~\ref{prop-bornology-comparison}, for any $C' > 0$ the set $Z_{C'} \coloneqq \{g \in G(L) \mid d(p_0,\iota(g)) < C' \}$ is a bounded set and for any bounded $Z \subset G(L)$ the constant $c_Z \coloneqq \sup \{d(p_0,\iota(y)) \mid y \in Z \}$ is finite.
 
The following lemma translates the results of \cite{CornutNicole:CrysBuildings} into our terms.

\begin{lemma} \label{lemma translation} Let $G$ be a reductive group over $F$ and $b\in G(L)$. 
\begin{assertionlist}
\item For any $c > 0$ there exists a $C>0$ such that if $x \in \Bcal^e(G,L)$ satisfies $d(x,b\sigma(x)) < c$ then there exists $x_0 \in \Bcal^e(J_b,F)$ with $d(x,x_0) < C$.
\item For any $c > 0$ there exists a $C>0$ such that if $x \in \iota(G(L))$ satisfies $d(x,b\sigma(x)) < c$ then there exists $x_0 \in \iota(J_b(F))$ with $d(x,x_0) < C$.
\end{assertionlist}
\end{lemma}
\begin{proof}
Assertion (a) is proven in \cite{CornutNicole:CrysBuildings} by an elegant geometrical argument. By Theorem~3.3 of loc.~cit.\ $f_\ast$ identifies $\Bcal^{e}(J_b,F)$ with the set
\[
 {\rm Min}(b\sigma) \coloneqq \{x \in \Bcal^e(G,L) \mid d(x,b\sigma(x)) \textnormal{ attains its minimal possible value.}\},
\]
Thus the statement (a) claims that if $d(x,b\sigma(x))$ is bounded, so is the distance to ${\rm Min}(b\sigma)$. This (together with an upper bound for $C$) is proven in \cite[Prop.~8]{CornutNicole:CrysBuildings}.

To show that (a) implies (b), we have to show that the distance of a point $x \in \Bcal^e(G,L)$ to $\iota(G(L))$ is bounded above, or equivalently that there exists a bounded subset $M \subset \Bcal^{e}(G,L)$ such that $G(L)\cdot M = \Bcal^{e}(G,L)$ as well as the analogous assertion for $J_b(F)$. For this, we fix an isomorphism $X_\ast(Z)^I \cong \ZZ^r$, which yields an identification $V_0(G,L) = \RR^r$. Then we may choose $M = \afr \times [0,1]^r$, where $\afr$ is any alcove of the usual Bruhat-Tits building $\Bcal(G,L)$.
\end{proof}

\begin{proof}[Proof of Proposition \ref{prop Jb-orbits}]
Let $Z \subset G(L)$ be bounded. We fix $g \in \Xtilde_Z(b)$ and denote $ x \coloneqq \iota(g)$. Then
 \[
  d(x,b\sigma(x)) = d(g^{-1}.x,g^{-1}.b\sigma(x)) = d(p_0,g^{-1}b\sigma(g).p_0) < c_Z.
 \]
 By Lemma \ref{lemma translation}(b), there exist a $C_Z > 0$ depending only on $Z$ and a $j \in J_b(F)$ such that
 \[
  d((j^{-1}\cdot g).p_0,p_0) = d(x, j.p_0) < C_Z,
 \]
 i.e.\ $j^{-1} \cdot g \in Z_{C_Z}$. Hence $\Xtilde_Z(b) = J_b(F) \cdot (\Xtilde_Z(b) \cap Z_{C_Z})$.
\end{proof}

It remains to prove the first part of Theorem~\ref{thm finiteness}. By Lemma~\ref{lemma-representability-in-Fl} it is equivalent to the following proposition.
 
 \begin{proposition} \label{prop bdd neighbourhood}
   Every $x_0 \in X_Z(b)(k)$ has a bounded open neighbourhood.
 \end{proposition}
 \begin{proof}
  The proof follows by an analogous argument as the last part of \cite[Thm.~6.3]{HartlViehmann:Newton}. Since the situation simplifies a lot by considering only the reduced structure, and since in loc.~cit.~only split groups, hyperspecial $P$, and certain $Z$ are considered, we give the complete proof for the reader's convenience.   
  
 Let $\omega = w_G(x_0)$ and let again $X_Z(b)^\omega=X_Z(b)\cap w_G^{-1}(\omega)$. Since $X_Z(b)^\omega \subset X_Z(b)$ is open and closed, it suffices to prove the claim for $X_Z(b)^\omega$. We can define a $G(L)$- and $\sigma$-invariant semi-metric $d\colon G(L)^\omega \to \NN \cup \{0\}$ by 
 \[
  d(g,h) \leq n \quad\Longleftrightarrow\quad h^{-1}g \in \overline {P \rho^{\vee}(\unif^{2n}) P} = \bigcup_{w \leq 2n\rho^{\vee}} PwP
 \] where $\rho^{\vee}$ denotes the half-sum of the positive coroots and $w\in  \Wtilde^P \backslash \Wtilde / \Wtilde^P$.
 Obviously this semi-metric descends to $\Flag_G^\omega$. Then a subset is bounded with respect to the bornology defined before Lemma~\ref{lemma-bornology-on-Fl} if and only if it is bounded with respect to $d$.
 
We choose $b$ as in Lemma~\ref{rem42}. Let $s \in \NN$ be as in the decency equation, i.e. $ (b\sigma)^s = (s\cdot\nu)(\unif)$. Enlarging $s$ and $Z$ if necessary, we assume that $\omega$ and $Z$ are both $\sigma^s$-invariant. Then $X_Z(b)^\omega$ is $\sigma^s$-stable and thus is defined over the extension $k_s$ of degree $s$ of $k_F$. The closed point $x_0$ defined over some finite extension of $k_F$. By enlarging $k_s$ further if necessary, we assume that $x_0$ is a $k_s$-rational point. We denote by $\Mcal$ the model of $X_Z(b)^\omega$ over $k_s$ and for every $n \in \NN$ we define the closed sub-ind-scheme
  \[
   \Mcal_n(k) \coloneqq  \{x \in \Mcal(k) \mid d(x,x_0) \leq n \}.
  \]
  Note that $\Mcal_n$ is actually a (perfectly) finite type scheme by Lemma~\ref{lemma-representability-in-Fl} and moreover defined over $k_s$ since $d$ is $\sigma$-invariant. Also note that
  $
   \Mcal = \varinjlim \Mcal_n.
  $
  
   The decency of $b$ implies that $J_b(F) \subset G(F_s)$ where $F_s$ is the unramified extension of $F$ of degree $s$. Thus the $J_b(F)^0$-action stabilises $\Mcal(k_s)$. Together with  Proposition~\ref{prop Jb-orbits} (which we proved independently of the first assertion of Theorem~\ref{thm finiteness}) we obtain that there exists an $N_0 \in \NN$ such that for every $x \in \Mcal(k)$ there exists a $y_0 \in \Mcal(k_s)$ with $d(x,y_0) \leq N_0$. For every $y_0 \in \Mcal(k_s)$ define the closed subscheme $\Mcal_n(y_0) \subset \Mcal_n$ by
   \[
    \Mcal_n(y_0)(k) \coloneqq \{y \in \Mcal_n(k) \mid d(y_0,y) \leq N_0 \}.
   \]
  Now consider the open subset of $\Mcal_n$
  \[
   U_n \coloneqq \Mcal_n(x_0) \setminus \bigcup_{y \in \Mcal(k_s) \atop d(x_0,y) > N_0} \Mcal_n(y).
  \]
  The union on the right hand side is indeed finite (and hence closed): By the triangular inequality $\Mcal_n(y)$ is empty unless $y \in \Mcal_{N_0+n}(k_s)$; the latter set is finite since $\Mcal_{N_0+n}$ is (perfectly) of finite type. We claim that the chain $U_1 \subseteq U_2 \subseteq \dotsm$ stabilises at $U_{2N_0}$ at the latest. To prove this, let $x \in U_n(k)$ for some $n$. We choose a rational point $y_0 \in \Mcal(k_s)$ with $d(x,y_0) \leq N_0$. By definition of $U_n$, we must have $d(x_0,y_0) \leq N_0$. Thus $d(x,x_0) \leq d(x,y_0) + d(y_0,x_0) \leq 2N_0$, i.e.\ $x \in U_{2N_0}$.
  
  Since $\Mcal = \varinjlim \Mcal_n$, the subset $U_{2N_0} = \varinjlim U_n$ is open in $\Mcal$. It is moreover bounded and contains $x_0$. It is thus a bounded open neighbourhood of $x_0$.
 \end{proof}

\def\cprime{$'$}
\providecommand{\bysame}{\leavevmode\hbox to3em{\hrulefill}\thinspace}
\providecommand{\MR}{\relax\ifhmode\unskip\space\fi MR }
\providecommand{\MRhref}[2]{%
  \href{http://www.ams.org/mathscinet-getitem?mr=#1}{#2}
}
\providecommand{\href}[2]{#2}

\end{document}